\documentclass[a4paper,11pt]{article}

\title{Generation of Iterated Wreath Products Constructed from
  Alternating, Symmetric and Cyclic Groups}
\author{Jiaping Lu \& Martyn Quick\\[10pt]
  Mathematical Institute, University of St Andrews,\\
  North Haugh, St Andrews, Fife, KY16 9SS\\[10pt]
  \begin{tabular}{c}\texttt{jl337@st-andrews.ac.uk}\\
    \texttt{jiapingljp@hotmail.com}
  \end{tabular}\qquad\qquad \texttt{mq3@st-andrews.ac.uk}}

\usepackage{amsmath,amssymb}
\usepackage[margin=2.9cm]{geometry}
\usepackage[amsmath,thmmarks]{ntheorem}
\usepackage{enumitem}
\usepackage{tikz}

\theorembodyfont{\slshape}
\newtheorem{lemma}{Lemma}[section]
\newtheorem{prop}[lemma]{Proposition}

\newtheorem{thm}[lemma]{Theorem}
\theorembodyfont{\normalfont}
\newtheorem{example}[lemma]{Example}

\theorembodyfont{\slshape}
\theoremnumbering{Alph}
\newtheorem{THM}{Theorem}
\newtheorem{COR}[THM]{Corollary}

\makeatletter
\theoremstyle{nonumberplain}
\theorembodyfont{\normalfont}
\theoremheaderfont{\normalfont\scshape}
\theoremsymbol{~\ensuremath\square}
\newtheoremstyle{proofstyle}%
  {\item[\theorem@headerfont\hskip\labelsep ##1\theorem@separator]}%
  {\item[\theorem@headerfont\hskip\labelsep ##3\theorem@separator]}
\theoremstyle{proofstyle}
\newtheorem{proof}{Proof:}
\makeatother

\newcommand{\0}{\mathbf{0}}
\newcommand{\trivsubgp}{\mathbf{1}}
\newcommand{\Cent}[2]{\mathrm{C}_{#1}(#2)}
\newcommand{\Cohom}[2][1]{\mathrm{H}^{#1}(#2)}
\newcommand{\Der}[1]{\mathrm{Der}(#1)}
\newcommand{\End}[1]{\operatorname{End}(#1)}
\newcommand{\Field}[1]{\mathbb{F}_{#1}}
\newcommand{\Frat}[1]{\Phi(#1)}
\newcommand{\Inn}[1]{\mathrm{Inn}(#1)}
\newcommand{\length}[1]{\mathopen{|}#1\mathclose{|}}
\newcommand{\order}[1]{\mathopen{|}#1\mathclose{|}}
\newcommand{\pr}[1]{\operatorname{pr}(#1)}
\newcommand{\rad}[1]{\operatorname{rad}#1}
\newcommand{\set}[2]{\{\,#1\mid#2\,\}}
\newcommand{\Zint}{\mathbb{Z}}

\renewcommand{\geq}{\geqslant}
\renewcommand{\leq}{\leqslant}

\renewcommand{\nleq}{\nleqslant}
\renewcommand{\wr}{\operatorname{wr}}

\newcommand{\nbd}{\nobreakdash-}

\setlist[enumerate,1]{label={\normalfont(\roman*)}}

\usepackage{hyperref}
\hypersetup{colorlinks=true, linkcolor={red!50!black},
  citecolor={green!50!black}}

\begin{document}

\maketitle

\begin{abstract}
  Let $G_{1}$,~$G_{2}$, \dots\ be a sequence of groups each of which
  is either an alternating group, a symmetric group or a cyclic
  group.  Let us construct a sequence~$(W_{i})$ of wreath products via
  $W_{1} = G_{1}$ and, for each $i \geq 1$, $W_{i+1} = G_{i+1} \wr
  W_{i}$ via the natural permutation action.  We determine the minimum
  number~$d(W_{i})$ of generators required for each wreath product in
  this sequence.
\end{abstract}



\section{Introduction}

One common way to describe a finite group is to provide a set of
generators.  In view of this, it is unsurprising that there have been
many studies into the minimum number of generators required for
particular groups.  Famous examples include the fact, often taught
within a first course on group theory, that a finite symmetric group
can be generated by two permutations, while it is a consequence of the
Classification of Finite Simple Groups that a finite simple group
needs only two generators.  (Steinberg~\cite{Steinberg} establishes
this for the finite simple groups of Lie type, while Aschbacher and
Guralnick complete the work by dealing with the sporadic simple
groups, see~\cite[Theorem~B]{AschGur}.)

Another theme within the study of generators for finite groups is the
behaviour of the minimum number of generators required for various
sequences of finite groups.  For example, in a sequence of papers
beginning with~\cite{Wiegold}, Wiegold considers the behaviour of the
number of generators required for a direct product~$G^{n}$ of
$n$~copies of a group~$G$ as $n \to \infty$.  An alternative
construction that could be used instead of the direct product is the
wreath product constructed from suitable actions of the groups
involved.  We shall describe a few of such studies that are most
closely linked to the content of this article.

Bhattacharjee~\cite{Bhatt} considers the probability of generating an
iterated wreath product of non-abelian simple alternating groups by
two randomly chosen elements, where these wreath products are
constructed from their natural actions.  She shows that this
probability can be made arbitrarily close to~$1$ by choosing the first
alternating group used to have large degree.  In particular, this
implies that such an iterated wreath product requires only two
generators: her result assumes that the first alternating group has
sufficiently large degree but this hypothesis is not actually needed
to reach this conclusion.  The second author~\cite{MRQ2} generalized
Bhattarcharjee's work and showed the same holds for an iterated wreath
product constructed using any non-abelian simple groups with any
faithful transitive actions.  Bondarenko~\cite{Bond} considers
iterated wreath products~$W_{n}$ of finite transitive permutation
groups~$G_{i}$ with uniformly bounded number of generators and shows
that the number of generators for~$W_{n}$ is bounded if and only if
the abelian groups $G_{1}/G_{1}' \times G_{2}/G_{2}' \times \dots
\times G_{n}/G_{n}'$ have bounded number of generators.

More recently, East and Mitchell~\cite{EastMitch} show, using
relatively elementary methods, that the wreath product of two finite
symmetric or alternating groups can be generated by two elements.
They suggest an investigation of the minimal number of generators for
an iterated wreath product involving alternating and symmetric
groups.  The technology to determine the number of generators required
for an iterated wreath product of symmetric groups has actually
existed for some time.  Lucchini~\cite[Theorem~8]{Lucc97} describes
the number of generators required for a wreath product of an abelian
group by a symmetric group of degree~$n \geq 2$.  We use this theorem
as part of our argument and this on its own would be enough to
determine the number of generators required for an iterated wreath
product of symmetric groups.  Finally, we mention that Woryna
describes the number of generators required for an iterated wreath
product of cyclic groups and later an iterated product of abelian
groups, see~\cite[Theorem~1.1]{Woryna11}
and~\cite[Theorem~1.1]{Woryna12}.  Both of Woryna's results also make
use of the work of Lucchini~\cite{Lucc97} to establish the bounds for
the number of generators required.

In this article, we shall establish the precise number of generators
required for an iterated wreath product constructed from alternating
groups, symmetric groups and cyclic groups in their natural action.
In the case of a cyclic group of order~$n$, by this we mean the
regular action of the cyclic group on $n$~points.  If one were to take
only cyclic groups, then one recovers Woryna's theorem
\cite[Theorem~1.1]{Woryna11}.  Taking only alternating groups of
degree at least~$5$, one recovers the observation that such an
iterated wreath product is $2$\nbd generated, while taking only
alternating groups and symmetric groups gives an answer to the
question raised by East and Mitchell.  Consequently, our main theorem
can be viewed as a simultaneous extension of several of the results
that we have just described.

To state our theorem, we first recall some standard notation.  If
$G$~is a finite group, we write~$d(G)$ for the minimum number of
generators required for~$G$.  If $A$~is a finite abelian group, we
write~$d_{p}(A)$ for the minimum number of generators required for the
Sylow $p$\nbd subgroup of~$A$.  Thus, it is a straightforward
consequence of the classification of finite abelian groups that, for
such an abelian group~$A$, \ $d(A) = \max d_{p}(A)$ where the maximum
is taken over all prime numbers~$p$.  We shall establish the following
theorem:

\begin{THM}
  \label{thm:main}
  Let $k \geq 2$ and let $G_{1}$,~$G_{2}$, \dots,~$G_{k}$ be a
  sequence of non-trivial finite groups each of which is either an
  alternating group, a symmetric group or a cyclic group.  Let $W =
  G_{k} \wr G_{k-1} \wr \dots \wr G_{1}$ be the iterated wreath
  product constructed from the natural action of each~$G_{i}$.  Define
  $A$~to be the abelianization of the iterated wreath product $G_{k}
  \wr G_{k-1} \wr \dots \wr G_{2}$.  Then
  \[
  d(W) = \begin{cases}
    \max (2, d(A), d_{3}(A)+1 ) &\text{if $G_{1} = A_{4}$}, \\
    \max (2, d(A) ) &\text{if $G_{1} = A_{n}$ with $n \geq 5$}, \\
    \max (2, d(A), d_{2}(A) + 1) &\text{if $G_{1} = S_{n}$ with $n \geq 3$}, \\
    \max (2, d(A) + 1 ) &\text{if $G_{1}$~is cyclic}.
  \end{cases}
  \]
\end{THM}

The abelian group~$A$ appearing in the above statement has the form
\[
A = G_{k}/G_{k}' \times G_{k-1}/G_{k-1}' \times \dots \times G_{2}/G_{2}'
\]
and we know that each factor~$G_{i}/G_{i}'$ is trivial when $G_{i}$~is
an alternative group of degree at least~$5$, is cyclic of order~$2$
when $G_{i}$~is a symmetric group of degree at least~$3$, is cyclic of
order~$3$ when $G_{i} \cong A_{4}$, and is isomorphic to the
original~$G_{i}$ when this is a cyclic group.  Consequently, one can
readily determine the values of~$d(A)$ and~$d_{p}(A)$ required in the
statement simply by counting the number of groups of each type
appearing in the sequence $G_{2}$,~$G_{3}$, \dots,~$G_{k}$.  The
following corollary is deduced by performing this elementary
arithmetic.  It notes that the number of generators required for the
iterated wreath product~$W$ can be expressed by the same formula
irrespective of the choice of~$G_{1}$ provided that this is a
non-abelian alternating or symmetric group.

\begin{COR}
  \label{cor:CountFormula}
  Let $k \geq 2$ and $G_{1}$,~$G_{2}$, \dots,~$G_{k}$ be a sequence of
  non-trivial finite groups each of which is either an alternating
  group, a symmetric group or a cyclic group.  Suppose that $G_{1}$~is
  \textbf{not} cyclic and define the following parameters:
  \begin{enumerate}
  \item $a_{4}$~is the number of terms $G_{1}$,~$G_{2}$,
    \dots,~$G_{k}$ that are alternating groups of degree~$4$;
  \item $s$~is the number of terms that are non-abelian symmetric
    groups; and,
  \item for each prime~$p$, $c_{p}$~is the number of terms that are
    cyclic groups of order divisible by~$p$.
  \end{enumerate}
  Let $W = G_{k} \wr G_{k-1} \wr \dots \wr G_{1}$ be the iterated
  wreath product constructed from the natural action of each~$G_{i}$.
  Then
  \[
  d(W) = \max(2, d(W/W')) = \max_{p} ( 2, c_{2}+s, c_{3}+a_{4}, c_{p} ),
  \]
  where the maximum is taken over all primes~$p$.
\end{COR}

Indeed, note that, expressed in terms of the parameters given in
Corollary~\ref{cor:CountFormula},
\[
d_{2}(W/W') = c_{2}+s, \qquad d_{3}(W/W') = c_{3}+a_{4}, \qquad
d_{p}(W/W') = c_{p}
\]
for $p \geq 5$.  Hence $d(W/W') = \max(c_{2}+s, c_{3}+a_{4}, c_{p})$
immediately.  The rest of the corollary is established by careful
consideration of the formula given in Theorem~\ref{thm:main}.

We finish this introduction by describing the structure of the paper.
In Section~\ref{sec:prelims}, we recall the definition of the wreath
product and fix the notation that we use.  We also state some results
that we depend upon, particularly including some important facts from
Lucchini's work~\cite{Lucc97}.  Our proof uses the methods
within~\cite{Lucc97}, but extends them to cover the case of a wreath
product of an abelian group by an alternating group.  We devote
Section~\ref{sec:AwrAn} to developing the necessary theory to
establish the number of generators of such a wreath product.  We prove
our main theorem in the final section of the article.  We do this by
determining a formula for the wreath product~$A \wr G_{1}$ and
establishing that $d(W) = \max(2, d(A \wr G_{1}))$, where $A$~and~$W$
are as in Theorem~\ref{thm:main} (see
Lemmas~\ref{lem:keylemma-abelian} and~\ref{lem:induction} below).

\paragraph{Acknowledgements:} The first author is funded by the China
Scholarship Council.  We thank the anonymous referee for their helpful
suggestions.

In order to meet institutional and research funder open access
requirements, any accepted manuscript arising shall be open access
under a Creative Commons Attribution (CC BY) reuse licence with zero
embargo.

\section{Preliminaries}
\label{sec:prelims}

In this section, we fix the notation that we use and summarize the
facts from previous work that we rely upon.  We start by recalling the
definition of a wreath product and associated notation.

Let $G$~and~$H$ be permutation groups on the finite sets $X$~and~$Y$,
respectively.  Let $B = \prod_{y \in Y} G_{y}$ be the direct product
of copies of~$G$ indexed by the set~$Y$.  We shall write $b =
(g_{y}) = (g_{y})_{y \in Y}$ for an element of~$B$ expressed as a
sequence of elements from~$G$ indexed by~$Y$.  We define an action
of~$H$ on~$B$ by permuting the entries of its elements in the same way
that $H$~acts on~$Y$:
\[
(g_{y})^{h} = (g_{yh^{-1}})
\]
for $b = (g_{y}) \in B$ and $h \in H$.  The (permutational) wreath
product $W = G \wr_{Y} H$ of~$G$ by~$H$ is the semidirect product $W =
B \rtimes H$ constructed via this action.  In this article, we shall
construct wreath products from natural actions specified for the
ingredients $G$~and~$H$.  Consequently, we shall typically omit the
subscript on the wreath product and simply write $W = G \wr H$ for the
wreath product of~$G$ by~$H$ constructed via the natural action of the
permutation group~$H$.

The wreath product $W = G \wr H$ has a natural action on the Cartesian
product~$X \times Y$ given by
\[
(x,y)^{g} = (x^{g_{y}}, y^{h})
\]
for $(x,y) \in X \times Y$ and $g = (g_{y})h \in W$.  As a
consequence, if $G_{1}$,~$G_{2}$, \dots,~$G_{k}$ is a sequence of
permutation groups acting on the sets $X_{1}$,~$X_{2}$,
\dots,~$X_{k}$, respectively, then we may iterate the wreath product
construction exploiting this induced action.  Thus we recursively
define $W_{1} = G_{1}$ acting upon the set~$X_{1}$ and then, having
defined~$W_{i-1}$ with an action on $X_{i-1} \times \dots \times
X_{1}$, we define $W_{i} = G_{i} \wr W_{i-1}$ with its action on
$X_{i} \times X_{i-1} \times \dots \times X_{1}$.  It is this iterated
wreath product that we employ in our main theorem.

Recall that $d(G)$~denotes the minimal number of generators of a
finite group~$G$.  We note the following important observation.

\begin{thm}[Lucchini--Menegazzo \cite{LuccMen}]
  \label{thm:LuccMen}
  If $G$~is a non-cyclic finite group with a unique minimal normal
  subgroup~$N$, then $d(G) = \max(2, d(G/N))$.
\end{thm}

As is common, if $G$~is any group, we shall use the term \emph{$G$\nbd
module} for an additively written abelian group~$M$ upon which
$G$~acts via automorphisms.  Such a $G$\nbd module can also be
considered as a module for the group ring~$\Zint G$.  If $M$~arises as
an abelian chief factor of the group~$G$, then it necessarily has
exponent~$p$ for some prime~$p$ and therefore can also be viewed as a
module for the group algebra~$\Field{p}G$ over the finite
field~$\Field{p}$ of $p$~elements.  We write~$I_{G}$ for the
\emph{augmentation ideal} of the group ring~$\Zint G$; that is, the
kernel of the natural map $\Zint G \to \Zint$ induced by mapping every
element of~$G$ to~$1$.

As with Lucchini's earlier work~\cite{Lucc97} on generation of wreath
products, we shall use the \emph{presentation rank}~$\pr{G}$ to
provide information about the number of generators of a group~$G$.
This is an invariant of the group that arises via the decomposition of
the relation module associated to a minimal free presentation of~$G$.
It follows from the definition that $\pr{G}$~is a non-negative
integer, though this can also be deduced from
Lemma~\ref{lem:Roggenkamp} below since it is straightforward to check
$d(I_{G}) \leq d(G)$.  In our work, we shall rely upon the following
collection of facts, the first of which allows us to link~$d(G)$ to
the number of generators of the augmentation ideal~$I_{G}$ as a $\Zint
G$\nbd module.

\begin{lemma}[Roggenkamp~\protect{\cite[Theorem~2.1]{Rogg}}]
  \label{lem:Roggenkamp}
  Let $G$~be a finite group.  Then $\pr{G} = d(G) - d(I_{G})$.
\end{lemma}

\begin{lemma}[Gruenberg~\cite{Gru76}]
  \label{lem:Gruenberg}
  Let $G$~be a finite group.
  \begin{enumerate}
  \item \label{i:quot-pr}
    {\normalfont \cite[(B.ii)]{Gru76}} If $\pr{G} > 0$ and $N$~is
    a soluble normal subgroup of~$G$, then $d(G) = d(G/N)$.
  \item \label{i:soluble-pr}
    {\normalfont \cite[(2.3)]{Gru76}} If $G$~is soluble, then $\pr{G}
    = 0$.
  \end{enumerate}
\end{lemma}

As noted in the introduction, we rely upon the methods and results
developed in Lucchini's work~\cite{Lucc97}.  We describe here the
required notation and, for convenience, recall the facts that we use.

A chief factor of a finite group~$G$ is the quotient~$H/K$ of normal
subgroups $H$~and~$K$ with $K < H$ and such that there is no normal
subgroup~$N$ of~$G$ satisfying $K < N < H$.  We shall say that
$H/K$~is a \emph{complemented} chief factor if there exists a
subgroup~$C$ such that $G = HC$ and $H \cap C = K$; that is, when
$C/K$~is a complement to the normal subgroup~$H/K$ in the quotient
group~$G/K$.  This concept arises in the following definition of a
parameter associated to the group~$G$ and a $G$\nbd module~$M$.  For
such a $G$\nbd module~$M$, we define $\delta_{G}(M)$~to be the number
of (abelian) chief factors of~$G$ that are isomorphic to~$M$ as
$G$\nbd modules and are complemented in a chief series for~$G$.  If an
abelian chief factor~$H/K$ of~$G$ is complemented, then the above
complement~$C$ may be taken to be a maximal subgroup of~$G$.  It then
follows by~\cite[Theorem~2.1]{Lafuente} that $\delta_{G}(M)$~does not
depend upon the choice of chief series for~$G$.

Recall, for example from~\cite[Chapter~11]{Robinson}, that if a
group~$G$ acts via automorphisms upon an additively-written abelian
group~$A$ then the first cohomology group~$\Cohom{G,A}$ is equal to
the quotient of the group~$\Der{G,A}$ of derivations $G \to A$ by the
group~$\Inn{G,A}$ of inner derivations.  A derivation $\delta \colon G
\to A$ is a function such that $(gh)^{\delta} = (g^{\delta})^{h} +
h^{\delta}$ for all $g,h \in G$, while an inner derivation is one of
the form $g \mapsto a - a^{g}$ for fixed $a \in A$.  We shall make use
of this technology when $A$~is a $G$\nbd module~$M$ that is isomorphic
to a chief factor of~$G$.

Such a $G$\nbd module $M$~is an irreducible $\Field{p}G$\nbd module
where $p$~is the exponent of the chief factor to which it is
isomorphic.  By Schur's Lemma, the endomorphism ring~$\End{M}$ is a
division ring containing all scalar multiples of the identity.  It is
therefore, by finiteness, isomorphic to some finite extension
of~$\Field{p}$.  The natural action of~$\End{M}$ upon~$M$ then endows
it with the structure of a vector space over this larger field.
Furthermore, the action of~$\End{M}$ induces an action on both the
groups $\Der{G,M}$ of derivations and $\Inn{G,M}$ of inner derivations
and hence the first cohomology group~$\Cohom{G,M}$ also has the
structure of a vector space over~$\End{M}$.  In this context,
Lucchini~\cite{Lucc97} defines the following parameters:
\begin{align}
  r_{G}(M) &= \dim_{\End{M}}M \notag \\
  s_{G}(M) &= \dim_{\End{M}}\Cohom{G,M} \notag \\
  h_{G}(M) &= \left\lfloor \frac{s_{G}(M)-1}{r_{G}(M)} \right\rfloor +
  2 \label{eq:h_G}
\end{align}
The first two parameters are the exponents arising when the orders
of~$M$ and of~$\Cohom{G,M}$ are expressed as powers
of~$\order{\End{M}}$.  For example, the formula
from~\cite[(1.2)]{Lucc97} that we quote in the next lemma arises by
considering the terms within~\cite[(2.10)]{AschGur} expressed as such
powers.  When we make use of these parameters in our work below, we
shall actually show that, for the $G$\nbd modules~$M$ that matter to
us, the endomorphism ring~$\End{M}$ actually consists precisely of the
scalar multiples of the identity.  Hence, in these cases,
$r_{G}(M)$~is simply the dimension of~$M$ and we estimate~$s_{G}(M)$
by bounding the number of derivations.

Finally in this section, we collect together from~\cite{Lucc97}
various facts that we shall use within the following portmanteau result.

\begin{lemma}[Lucchini~\cite{Lucc97}]
  \label{lem:Lucchini-facts}
  \begin{enumerate}
  \item \label{i:Lucc-augment}
    {\normalfont \cite[Proposition~1]{Lucc97}}
    Let $H$~be a finite group and $G$~be a transitive permutation
    group of degree~$n$.  Then
    \[
    d(I_{H \wr G}) = \max \left( d( I_{(H/H') \wr G}), \left\lfloor
    \frac{d(I_{H}) - 2}{n} \right\rfloor + 2 \right).
    \]
  \item \label{i:Lucc-solwreath}
    {\normalfont \cite[Theorem~2]{Lucc97}}
    Let $H$~be a finite soluble group and $G$~be a transitive
    permutation group of degree~$n$.  Then
    \[
    d(H \wr G) = \max \left( d( (H/H') \wr G ), \left\lfloor
    \frac{d(H) - 2 }{n} \right\rfloor + 2 \right).
    \]

  \item \label{i:Lucchini-s}
    {\normalfont \cite[(1.2)]{Lucc97}}
    Let $M$~be an irreducible $G$\nbd module for a finite group~$G$.
    Then
    \[
    s_{G}(M) = \delta_{G}(M) + \dim_{\End{M}} \Cohom{G/\Cent{G}{M},M}.
    \]

  \item \label{i:Lucchini-commquot}
    {\normalfont \cite[(1.5)]{Lucc97}}
    Let $N$~be a normal subgroup of a finite group~$G$ with $N \leq
    G'$.  Then
    \[
    d(I_{G}) \leq \max_{M} (2, d(I_{G/N}), h_{G}(M) )
    \]
    where $M$~ranges over the non-trivial irreducible $G$\nbd modules
    such that $\delta_{G/N}(M) < \delta_{G}(M)$.
  \end{enumerate}
\end{lemma}

In the third part of this lemma, the centralizer~$\Cent{G}{M}$ is the
kernel of the action of~$G$ upon the irreducible module~$M$.  Hence
$G/\Cent{G}{M}$~is the group of automorphisms of~$M$ induced by~$G$.
Its role in this formula is explained in~\cite[(2.10)]{AschGur}.  Note
also that the condition $\delta_{G/N}(M) < \delta_{G}(M)$ appearing
within part~\ref{i:Lucchini-commquot} implies that when we refine the
series $\trivsubgp \leq N \leq G$ to a chief series for~$G$, not all
the complemented abelian chief factors that are isomorphic to~$M$ as
$G$\nbd modules appear between $N$~and~$G$ and hence at least one
occurs within~$N$.

\section{\boldmath Wreath products~$A \wr A_{n}$ with abelian base
  group}
\label{sec:AwrAn}

In Lucchini~\cite[Theorem~8]{Lucc97}, a formula is given for the
number of generators of a wreath product~$A \wr S_{n}$ of an abelian
group~$A$ by a symmetric group of degree~$n$.  We use this formula to
establish part of Lemma~\ref{lem:keylemma-abelian} in the next
section.  We shall also need to establish an analogous formula for the
number of generators of a wreath product~$A \wr A_{n}$ of an abelian
group~$A$ by an alternating group of degree~$n$.  This purpose of this
section is to adapt the methods of Lucchini~\cite{Lucc97} to
determine~$d(A \wr A_{n})$ as given in
Propositions~\ref{prop:wreath-A4} and~\ref{prop:wreath-An} below.

Fix an integer $n \geq 4$.  If $p$~is a prime number, define $V =
\Field{p}^{n}$ to be the permutation module over the field~$\Field{p}$
associated to the natural action of the alternating group~$A_{n}$ on
$n$~points.  Let $I_{p}$~denote the kernel of the map $V \to
\Field{p}$ given by $(x_{1},x_{2},\dots,x_{n}) \mapsto \sum_{i=1}^{n}
x_{i}$.  For every pair $i$~and~$j$ with $1 \leq i < j \leq n$, let
$e_{ij} = (y_{1},y_{2},\dots,y_{n})$ where $y_{i} =1$, $y_{j} = -1$
and $y_{k} = 0$ for $k \neq i,j$.  Observe that, since $A_{n}$~acts
$2$\nbd transitively on $\Omega = \{1,2,\dots,n\}$, each~$e_{ij}$ is a
generator for~$I_{p}$ as an $\Field{p}A_{n}$\nbd module.

\begin{lemma}
  \label{lem:Ip}
  \begin{enumerate}
  \item \label{i:Ip-p|n}
    If $p$~divides~$n$, then $I_{p}$~is the unique maximal submodule
    of the permutation module~$V$.
  \item \label{i:Ip-coprime}
    If $p$~does not divide~$n$, then $V \cong I_{p} \oplus Z_{p}$
    where $Z_{p}$~is a $1$\nbd dimensional trivial module.
    Furthermore, in this case, $I_{p}$~is an irreducible
    $\Field{p}A_{n}$\nbd module with $\End{I_{p}} \cong \Field{p}$ and
    therefore $r_{A_{n}}(I_{p}) = n-1$.
  \end{enumerate}
\end{lemma}

\begin{proof}
  \ref{i:Ip-p|n}~Since $V/I_{p}$~is isomorphic to the trivial
  module~$\Field{p}$, certainly $I_{p}$~is a maximal submodule of~$V$.
  To prove uniqueness, it is sufficient to show that any element
  outside~$I_{p}$ generates the module~$V$.  Let $x =
  (x_{1},x_{2},\dots,x_{n}) \in V \setminus I_{p}$ and let $U$~be the
  submodule generated by~$x$.  Since $p \mid n$, there exist
  $i$~and~$j$ such that $x_{i} \neq x_{j}$. 
  
  We first consider the case that $n = 4$. Necessarily then $p = 2$
  and we may assume either $x = (1,0,0,0)$ or $x = (0,1,1,1)$ since
  $A_{4}$~acts transitively on~$\Omega$.  Certainly $(1,0,0,0)$
  generates~$V$.  If $x = (0,1,1,1)$, then $x^{(1\;2\;3)} - x =
  e_{12}$, which generates~$I_{p}$, and therefore, by maximality
  of~$I_{p}$, the submodule~$U$ generated by~$x$ equals~$V$.

  For the remainder of the proof, assume then that $n \geq 5$.  By the
  $2$\nbd transitivity of~$A_{n}$, there is no loss of generality in
  assuming that $x_{1} \neq x_{2}$.  Now let
  \[
  y = x - x^{(1\;2)(3\;4)} = (\alpha, -\alpha, \beta, -\beta, 0,
  \dots, 0)
  \]
  for $\alpha = x_{1} - x_{2} \neq 0$ and $\beta = x_{3} -
  x_{4}$. Then $z = y - y^{(1\;2\;5)} =
  (\alpha,-2\alpha,0,0,\alpha,0,\dots,0)$ and, dividing by~$\alpha$,
  we may assume $z = (1,-2,0,0,1,0,\dots,0) \in U$.  Finally
  $z^{(1\;3\;5)} - z = e_{35}$, which generates~$I_{p}$.  We therefore
  conclude that $U = V$, as required.

  \ref{i:Ip-coprime}~Define $Z_{p} = \set{(x, x, \dots , x)}{x \in
    \Field{p}}$.  Then $Z_{p}$~is a one-dimensional submodule of~$V$
  upon which $A_{n}$~acts trivially and $Z_{p} \cap I_{p} = \0$ as
  $p$~does not divide~$n$.  Hence $V = I_{p} \oplus Z_{p}$.  If $x =
  (x_{1},x_{2},\dots,x_{n})$ is a non-zero element of~$I_{p}$, then $x
  \notin Z_{p}$ so $x_{i} \neq x_{j}$ for some $i$~and~$j$.  Then a
  similar approach to that used in part~\ref{i:Ip-p|n} shows that
  $x$~generates~$I_{p}$ and it follows that $I_{p}$~is irreducible.
  
  Fix $a = (1, -1, 1, -1, 0, \dots, 0)$.  This is then a generator
  for~$I_{p}$, so any endomorphism of the module~$I_{p}$ is determined
  by its effect on~$a$.  Let $\phi \in \End{I_{p}}$ and set $b =
  (b_{1},b_{2},\dots,b_{n}) = a\phi$.  We consider first the case when
  $p$~is odd.  Applying~$\phi$ to $a^{(1\;2)(3\;4)} = -a$ shows that
  $b_{2} = -b_{1}$ and $b_{4} = -b_{3}$ and $b_{i} = 0$ for $i \geq 5$,
  while applying~$\phi$ to $a^{(1\;3)(2\;4)} = a$ shows that $b_{1} =
  b_{3}$.  Hence $a\phi = b_{1}a$ and so $\phi$~is multiplication by
  the scalar~$b_{1}$.

  Now suppose $p = 2$ so that $a = (1,1,1,1,0,\dots,0)$ and $n$~is an
  odd number with $n \geq 5$.  Using the same approach, it follows
  that $b_{1} = b_{2} = b_{3} = b_{4}$.  Let $\sigma$~be the
  $(n-4)$\nbd cycle $(5 \; 6 \, \dots \, n)$.  Applying~$\phi$ to the
  equation $a^{\sigma} = a$ shows that $b_{5} = b_{6} = \dots =
  b_{n}$.  Since $b \in I_{p}$, we deduce $(n-4)b_{5} = 0$ and hence
  $b_{5} = b_{6} = \dots = b_{n} = 0$ since $p = 2$ and $n$~is odd.
  Thus $a\phi = b_{1}a$ and, again, $\phi$~is multiplication
  by~$b_{1}$.

  Now we have shown that the endomorphism ring of~$I_{p}$ consists
  merely of the scalars, it follows that $r_{A_{n}}(I_{p}) =
  \dim_{\Field{p}}(I_{p}) = n-1$.
  \end{proof}

The method used to determine the number of generators of a wreath
product $W = A \wr A_{n}$ of an abelian group~$A$ by an alternating
group will involve taking the quotient by the Frattini
subgroup~$\Frat{B}$ of the base group~$B$ of~$W$.  The
quotient~$W/\Frat{B}$ is isomorphic to the wreath product of the
quotient~$A/\Frat{A}$ by~$A_{n}$.  Accordingly, we now describe
properties of the chief factors of such a quotient.

\begin{lemma}
  \label{lem:chief-Ip}
  Let $\bar{A}$~be a direct product of elementary abelian groups and
  let $\bar{W} = \bar{A} \wr A_{n}$ be the wreath product with respect
  to the natural action of the alternating group~$A_{n}$.  Let $X$~be
  a complemented chief factor of~$\bar{W}$ contained in the base
  group~$\bar{B}$ such that $A_{n}$~acts non-trivially on~$X$.  Then
  $X$~is isomorphic to~$I_{p}$ for some prime divisor~$p$ of the
  order of~$\bar{A}$ with $p \nmid n$ and $\delta_{\bar{W}}(X) =
  \delta_{\bar{W}}(I_{p}) = d_{p}(\bar{A})$.
\end{lemma}

\begin{proof}
  We shall view the base group~$\bar{B}$ of~$\bar{W}$ as a module
  for~$A_{n}$.  Its radical~$\rad{\bar{B}}$ is, by definition, the
  intersection of the maximal submodules of~$\bar{B}$ and consequently
  corresponds to a normal subgroup of~$\bar{W}$.  We shall refine the
  series
  \begin{equation}
    \trivsubgp \leq \rad{\bar{B}} < \bar{B} < \bar{W}
    \label{eq:W-normalseries}
  \end{equation}
  to a chief series for~$\bar{W}$.  Hence if $X = H/K$ is a chief
  factor of~$\bar{W}$ with $H \leq \bar{B}$, then we may assume that
  either $\rad{\bar{B}} \leq K < H \leq \bar{B}$ or $K < H \leq
  \rad{\bar{B}}$.

  Consider first a chief factor $X = H/K$ with $K < H \leq
  \rad{\bar{B}}$ and suppose that $X$~is complemented.  Then there is
  a subgroup~$C$ of~$\bar{W}$ such that $HC = \bar{W}$ and $C \cap H =
  K$.  Then $\bar{B} = \bar{B} \cap HC = (\bar{B} \cap C)H$ and
  $\bar{B} \cap C \cap H = K$, so $\bar{B}/K = (\bar{B} \cap C)/K
  \oplus H/K$ as modules.  Since $H/K$~is a minimal submodule
  of~$\bar{B}/K$, necessarily $(\bar{B} \cap C)/K$~is a maximal
  submodule of~$\bar{B}/K$.  Hence $H \leq \rad{\bar{B}} \leq \bar{B}
  \cap C$, which implies $K = \bar{B} \cap C \cap H = H$ and this is
  impossible.
  
  Hence if $X = H/K$ is a complemented chief factor of~$\bar{W}$
  contained in~$\bar{B}$, then we may assume $\rad{\bar{B}} \leq K < H
  \leq \bar{B}$.  As a module, $\bar{B}$~is the direct sum of $p$\nbd
  primary submodules~$\bar{B}_{p}$ for each prime~$p$ dividing the
  order of~$\bar{A}$ and hence
  \[
  \rad{\bar{B}} = \bigoplus_{p \mid \order{\bar{A}}} \rad{\bar{B}_{p}}.
  \]
  Furthermore, by construction $\bar{B}_{p}$~is a direct sum of
  $d_{p}(\bar{A})$~copies of the permutation module $V =
  \Field{p}^{n}$ associated to the natural action of~$A_{n}$ and
  hence, by Lemma~\ref{lem:Ip}, $\rad{\bar{B}_{p}} = \0$ when $p$~does
  not divide~$n$ and $\rad{\bar{B}_{p}} = I_{p}^{d_{p}(\bar{A})}$ when
  $p$~divides~$n$.  Therefore
  \begin{equation}
    \bar{B}/\rad{\bar{B}} \cong
    \biggl( \bigoplus_{p \mid \order{\bar{A}}} Z_{p}^{d_{p}(\bar{A})}
    \biggr) \oplus
    \biggl( \bigoplus_{\substack{p \mid \order{\bar{A}},\\p \nmid n}}
    I_{p}^{d_{p}(\bar{A})} \biggr).
    \label{eq:B/radB}
  \end{equation}
  Here each~$Z_{p}$ is trivial as a $\Field{p}A_{n}$\nbd module.
  Hence, if $X$~is a complemented chief factor of~$\bar{W}$ contained
  in~$\bar{B}$ that is not trivial as a module, then $X \cong I_{p}$
  for some prime divisor~$p$ of the order of~$\bar{A}$ with $p \nmid
  n$.  Furthermore, $I_{p}$~occurs precisely $d_{p}(\bar{A})$~many
  times as a composition factor of the
  module~$\bar{B}/{\rad{\bar{B}}}$ and, by use of the direct sum
  decomposition in Equation~\eqref{eq:B/radB}, one can observe that
  each factor is complemented as a chief factor of~$\bar{W}$.

  If $X = H/K$~is a factor of the chief series obtained by
  refining~\eqref{eq:W-normalseries} with $H \nleq \bar{B}$, then
  $X$~corresponds to a chief factor of $A_{n} \cong \bar{W}/\bar{B}$.
  If $n \geq 5$, then $X \cong A_{n}$ and this is non-abelian.  If $n
  = 4$ the only chief factor of~$A_{4}$ upon which~$A_{4}$ acts
  non-trivially is isomorphic to~$C_{2} \times C_{2}$ and hence not
  isomorphic to~$I_{p}$ as $p$~is odd in this case.  We conclude that
  any chief factor isomorphic to~$I_{p}$ is contained in~$\bar{B}$ and
  therefore $I_{p}$~occurs precisely $d_{p}(\bar{A})$~times as a
  complemented chief factor of~$\bar{W}$, namely corresponding to the
  composition factors of the module~$\bar{B}/{\rad{\bar{B}}}$.
\end{proof}

We continue the discussion of complemented chief factors in the wreath
product $W = A \wr A_{n}$ that we began in Lemma~\ref{lem:chief-Ip}.
We concentrate first on the case $n = 4$.

\begin{lemma}
  \label{lem:h4-big-p}
  Let $\bar{A}$~be a direct product of elementary abelian groups, let
  $\bar{W} = \bar{A} \wr A_{4}$ be the wreath product of~$\bar{A}$ by
  the alternating group of degree~$4$ with respect to its natural
  action and let $p$~be a prime with $p \geq 5$.  Then
  \begin{enumerate}
  \item \label{i:H^1-4}
    $\order{\Cohom{A_{4},I_{p}}} \leq p$;
  \item \label{i:s4}
    $s_{\bar{W}}(I_{p}) \leq d_{p}(\bar{A}) + 1$;
  \item \label{i:h4}
    $h_{\bar{W}}(I_{p}) \leq \max(2,d_{p}(\bar{A}))$.
  \end{enumerate}
\end{lemma}

\begin{proof}
  \ref{i:H^1-4}~An element of the permutation module~$V =
  \Field{p}^{4}$ is fixed by all elements of~$A_{4}$ if and only if it
  has the form $v = (x,x,x,x)$ for some $x \in \Field{p}$.  Since
  $p$~does not divide $n = 4$, such an element lies in~$I_{p}$ only
  when $x = 0$.  Hence $\Cent{I_{p}}{A_{4}} = \set{ a \in I_{p} }{
    \text{$a^{g} = a$ for all $g \in A_{4}$} }$ contains only zero and
  we conclude that $\order{\Inn{A_{4},I_{p}}} = \order{I_{p}} = p^{3}$.

  If $\delta \colon A_{4} \to I_{p}$ is any derivation, then it is
  determined by its values on the permutations $(1\;2\;3)$ and
  $(1\;2)(3\;4)$ since these generate~$A_{4}$.  Let
  $(x_{1},x_{2},x_{3},x_{4}) = (1\;2\;3)^{\delta}$.  Applying~$\delta$
  to the formula $(1\;2\;3)^{3} = 1$ yields
  \[
  (x_{1},x_{2},x_{3},x_{4}) + (x_{1},x_{2},x_{3},x_{4})^{(1\;2\;3)} +
  (x_{1},x_{2},x_{3},x_{4})^{(1\;3\;2)} = 0
  \]
  and hence $x_{1}+x_{2}+x_{3} = 0$ and $3x_{4} = 0$.  Since $p \neq
  3$, we conclude there at most $p^{2}$~choices
  for~$(1\;2\;3)^{\delta}$.  Similarly, there are at most
  $p^{2}$~choices for $(y_{1},y_{2},y_{3},y_{4}) = \bigl( (1\;2)(3\;4)
  \bigr)^{\delta}$ since the same method shows that $y_{1}+y_{2} =
  y_{3}+y_{4} = 0$.  Hence $\order{\Der{A_{4},I_{p}}} \leq p^{4}$.  It
  now follows that $\order{\Cohom{A_{4},I_{p}}} =
  \order{\Der{A_{4},I_{p}}}/\order{\Inn{A_{4},I_{p}}} \leq p$.

  \ref{i:s4}~As observed in Lemma~\ref{lem:chief-Ip},
  $\delta_{\bar{W}}(I_{p}) = d_{p}(\bar{A})$.  Since $A_{4}$~acts
  faithfully on~$I_{p}$, the centralizer of~$I_{p}$ in~$\bar{W}$
  equals the base group.  Note that $\End{I_{p}} = \Field{p}$ by
  Lemma~\ref{lem:Ip}\ref{i:Ip-coprime} and hence, by
  Lemma~\ref{lem:Lucchini-facts}\ref{i:Lucchini-s},
  \[
  s_{\bar{W}}(I_{p}) = d_{p}(\bar{A}) + \dim_{\Field{p}}\Cohom{A_{4},I_{p}}
  \]
  and therefore the claimed bound for~$s_{\bar{W}}(I_{p})$ follows
  from part~\ref{i:H^1-4}.

  \ref{i:h4}~Note that the action of~$\bar{W}$ on~$I_{p}$ is induced
  by the action of the quotient~$A_{4}$ and consequently the structure
  of~$I_{p}$ is unchanged whether we view it as an $A_{4}$\nbd module
  or a $\bar{W}$\nbd module.  Hence, by use of
  Lemma~\ref{lem:Ip}\ref{i:Ip-coprime}, $r_{\bar{W}}(I_{p}) =
  r_{A_{4}}(I_{p}) = 3$ and substituting this and our bound in the
  previous part into the Formula~\eqref{eq:h_G}
  for~$h_{\bar{W}}(I_{p})$ yields
  \[
  h_{\bar{W}}(I_{p}) = \left\lfloor
  \frac{s_{\bar{W}}(I_{p})-1}{r_{\bar{W}}(I_{p})} \right\rfloor + 2 \leq
  \left\lfloor{\textstyle\frac{1}{3}}d_{p}(\bar{A})\right\rfloor + 2
  \leq \max(2,d_{p}(\bar{A})),
  \]
  as required.
\end{proof}

The argument when $p = 3$ requires a little adjustment.  The
constraint $3x_{4} = 0$ on the image of $(1\;2\;3)$ under a derivation
$\delta \colon A_{4} \to I_{3}$ yields no information in this case.
One may therefore only conclude that $\order{\Cohom{A_{4},I_{3}}} \leq
3^{2}$.  Following this adjustment, the result is as follows:

\begin{lemma}
  \label{lem:h4-p=3}
  Let $\bar{A}$~be a direct product of elementary abelian groups and
  $\bar{W} = \bar{A} \wr A_{4}$ be the wreath product of~$\bar{A}$ by
  the alternating group of degree~$4$ with respect to its natural
  action.  Then
  \[
  h_{\bar{W}}(I_{3}) \leq \max(2, d_{3}(\bar{A})+1).
  \]
\end{lemma}

We may now put the above information together to determine the minimum
number of generators of a wreath product $A \wr A_{4}$ where $A$~is a
finite abelian group.

\begin{prop}
  \label{prop:wreath-A4}
  Let $A$~be a finite abelian group and define $W = A \wr A_{4}$ to be
  the wreath product of~$A$ by~$A_{4}$ with respect to the natural
  action of~$A_{4}$ on $4$~points.  Then
  \[
  d(W) = \max (2, d(A), d_{3}(A)+1 ).
  \]
\end{prop}

\begin{proof}
  Define $d = \max(2, d(A), d_{3}(A)+1)$.  Since there is a surjective
  homomorphism from~$W$ onto~$A \times A_{4}$ and the latter has a
  homomorphism onto~$A \times C_{3}$, it follows that $d(W) \geq d$.
  
  Write $B$~for the base group of the wreath product~$W$ and $\bar{A}
  = A/\Frat{A}$ for the quotient of the abelian group~$A$ by its
  Frattini subgroup~$\Frat{A}$.  Since $B$~is a normal subgroup
  of~$W$, the Frattini subgroup of~$B$ is contained in that of~$W$.
  Therefore
  \[
  d(W) = d(W/\Frat{B}) = d(\bar{W})
  \]
  where $\bar{W} = W/\Frat{B} \cong \bar{A} \wr A_{4}$.  Let us
  write~$\bar{B}$ for the base group of~$\bar{W}$; that is, the image
  of~$B$ under the natural map $W \to \bar{W}$.  Since $\bar{W}$~is
  certainly soluble, Lemma~\ref{lem:Gruenberg}\ref{i:soluble-pr} tells
  us that it has zero presentation rank and hence
  \[
  d(\bar{W}) = d(I_{\bar{W}})
  \]
  by Lemma~\ref{lem:Roggenkamp}.  We now apply
  Lemma~\ref{lem:Lucchini-facts}\ref{i:Lucchini-commquot} with $N =
  [\bar{B},A_{4}]$, noting that then $\bar{W}/N \cong \bar{A} \times
  A_{4}$, to deduce that
  \[
  d(I_{\bar{W}}) \leq \max_{M} (2, d(I_{\bar{A} \times A_{4}}),
  h_{\bar{W}}(M) )
  \]
  where the maximum is taken over all non-trivial irreducible
  $\bar{W}$\nbd modules with $\delta_{\bar{W}/N}(M) <
  \delta_{\bar{W}}(M)$.  If $M$~is a $\bar{W}$\nbd module with
  $\delta_{\bar{W}/N}(M) < \delta_{\bar{W}}(M)$, then $M$~occurs as
  some complemented chief factor contained in~$N$ and hence
  in~$\bar{B}$.  Lemma~\ref{lem:chief-Ip} tells us that such an~$M$ is
  isomorphic to~$I_{p}$ for some odd prime divisor~$p$ of the
  order~$A$.  In addition, a further application of
  Lemma~\ref{lem:Roggenkamp} tells us that $d(I_{\bar{A} \times
    A_{4}}) = d(\bar{A} \times A_{4}) \geq 2$.  Putting this together,
  we deduce that
  \begin{equation}
    d(W) = d(I_{\bar{W}}) \leq \max_{p} ( d(\bar{A} \times A_{4}),
    h_{\bar{W}}(I_{p}) )
    \label{eq:d(W)-bound-A4}
  \end{equation}
  where the maximum is taken over all odd prime divisors
  of~$\order{A}$.  By Lemmas~\ref{lem:h4-big-p} and~\ref{lem:h4-p=3},
  $h_{\bar{W}}(I_{p}) \leq d$.

  Since $d \geq d(A)$, we may find a generating set for~$\bar{A}$
  consisting of $d$~elements, say $\bar{A} = \langle
  x_{1},x_{2},\dots,x_{d} \rangle$.  Moreover, since $d > d_{3}(A)$,
  we may assume that the order of~$x_{1}$ is coprime to~$3$.  Define
  $y_{1} = (x_{1},(1\;2\;3))$, \ $y_{2} = (x_{2},(1\;2)(3\;4))$ and
  $y_{i} = (x_{i},1)$ for $i \geq 3$.  Set $H = \langle
  y_{1},y_{2},\dots,y_{d} \rangle$.  Then $H$~contains $(1,(1\;2\;3))$
  as a power of~$y_{1}$ and it contains $[y_{1},y_{2}] =
  (1,(1\;3)(2\;4))$.  Since these generate the copy of~$A_{4}$, it now
  follows that $H = \bar{A} \times A_{4}$.  Hence $d(\bar{A} \times
  A_{4}) \leq d$.  Substituting this into
  Equation~\eqref{eq:d(W)-bound-A4} completes the verification that
  $d(W) \leq d$ and establishes the proposition.
\end{proof}

We now turn to consider the wreath product $W = A \wr A_{n}$ of a
finite abelian group~$A$ by an alternating group of degree~$n \geq
5$.  The argument runs broadly parallel to that for the case $n = 4$.

\begin{lemma}
  \label{lem:h-for-An}
  Let $\bar{A}$~be a direct product of elementary abelian groups and
  let $n \geq 5$.  Let $\bar{W} = \bar{A} \wr A_{n}$ be the wreath
  product of~$\bar{A}$ by the alternating group of degree~$n$ with
  respect to its natural action and let $p$~be a prime that does not
  divide~$n$.  Then
  \begin{enumerate}
  \item \label{i:H1-arb}
    $\order{\Cohom{A_{n},I_{p}}} \leq p^{2}$;
  \item \label{i:s-arb}
    $s_{\bar{W}}(I_{p}) \leq d_{p}(\bar{A}) + 2$;
  \item \label{i:h-arb}
    $h_{\bar{W}}(I_{p}) \leq \max(2, d_{p}(\bar{A}) )$.
  \end{enumerate}
\end{lemma}

\begin{proof}
  The method is similar to Lemma~\ref{lem:h4-big-p}.  For
  part~\ref{i:H1-arb}, we determine an upper bound for the number of
  derivations $\delta \colon A_{n} \to I_{p}$.  Observe that
  $A_{n}$~can be generated by $(1\;2)(3\;4)$ together with the cycle
  $(2\;4\;5\,\dots\,n)$ or $(2\;3\;4\,\dots\,n)$, depending upon
  whether $n$~is odd or even.  A derivation~$\delta$ is then
  determined by its images on the chosen two generators.  We shall
  find a restriction for the number of choices for the image $x =
  (x_{1},x_{2},\dots,x_{n}) = \bigl((1\;2)(3\;4)\bigr)^{\delta}$.
  Applying~$\delta$ to the formula $\big( (1\;2)(3\;4) \bigr)^{2} = 1$
  yields the equations
  \[
  x_{1} + x_{2} = x_{3} + x_{4} = 2x_{5} = 2x_{6} = \dots = 2x_{n} = 0.
  \]
  Hence if $p \neq 2$, there are at most $p^{2}$~choices for~$x$.

  Let us assume that $p = 2$, so necessarily $n$~is odd.  If
  $\sigma$~is any element of~$A_{n}$ that fixes the points
  in~$\{1,2,3,4\}$, then upon applying~$\delta$ to the formula
  $(1\;2)(3\;4) \, \sigma = \sigma \, (1\;2)(3\;4)$, we deduce that
  \[
  x^{\sigma} + \sigma^{\delta} = (\sigma^{\delta})^{(1\;2)(3\;4)} + x.
  \]
  Hence $x_{i\sigma} = x_{i}$ for all $i \geq 5$ as
  $(1\;2)(3\;4)$~fixes the corresponding coordinate
  of~$\sigma^{\delta}$.  Since $n \neq 6$, the pointwise stabilizer of
  $\{1,2,3,4\}$ in~$A_{n}$ is transitive on the remaining $n-4$~points
  of~$\Omega$.  This gives us enough freedom in the choice of~$\sigma$
  to deduce that $x_{5} = x_{6} = \dots = x_{n}$.  Combining this with
  the facts that $x \in I_{p}$ and $x_{1}+x_{2} = x_{3}+x_{4} = 0$
  yields $(n-4)x_{5} = 0$.  Thus $x_{5} = x_{6} = \dots = x_{n} = 0$
  since $n-4$~is odd and $p = 2$.  In conclusion, when $p = 2$, there
  are also at most $p^{2}$~choices for $x = \bigl( (1\;2)(3\;4)
  \bigr)^{\delta}$.

  The image of the other chosen generator for~$A_{n}$ under the
  derivation~$\delta$ is some element of~$I_{p}$.  We therefore deduce
  that $\order{\Der{A_{n},I_{p}}} \leq p^{n+1}$.  By the same argument
  as in Lemma~\ref{lem:h4-big-p}, $\order{\Inn{A_{n},I_{p}}} =
  \order{I_{p}} = p^{n-1}$ and therefore $\order{\Cohom{A_{n},I_{p}}}
  \leq p^{2}$.  By Lemma~\ref{lem:Ip}\ref{i:Ip-coprime}, $\End{I_{p}}
  \cong \Field{p}$ and $r_{\bar{W}}(I_{p}) = n-1$.  Part~\ref{i:s-arb}
  then follows from Lemma~\ref{lem:Lucchini-facts}\ref{i:Lucchini-s},
  while part~\ref{i:h-arb} follows from the formula~\eqref{eq:h_G}
  for~$h_{\bar{W}}(I_{p})$ and the fact that $n \geq 5$.
\end{proof}

We now complete the section by determining the number of generators of
a wreath product~$A \wr A_{n}$ of an abelian group~$A$ by an
alternating group of degree $n \geq 5$.

\begin{prop}
  \label{prop:wreath-An}
  Let $A$~be a finite abelian group and $n$~be an integer with $n \geq
  5$.  Define $W = A \wr A_{n}$ be the wreath product of~$A$
  by~$A_{n}$ with respect to the natural action of~$A_{n}$ on
  $n$~points.  Then
  \[
  d(W) = \max(2, d(A)).
  \]
\end{prop}

\begin{proof}
  The method is similar to that used in
  Proposition~\ref{prop:wreath-A4}.  Define $d = \max (2,d(A))$.
  First $d(W) \geq d$ since there is a surjective homomorphism onto $A
  \times A_{n}$.  For the reverse inequality, write $B$~for the base
  group of~$W$, \ $\bar{A} = A/\Frat{A}$ and $\bar{W} = W/\Frat{B} =
  \bar{A} \wr A_{n}$.  If it were the case that $\pr{\bar{W}} > 0$,
  then take $N = \bar{B} = B/\Frat{B}$ in
  Lemma~\ref{lem:Gruenberg}\ref{i:quot-pr} to see that $d(W) =
  d(\bar{W}) = d(\bar{W}/\bar{B}) = 2$ and the reverse inequality
  would follow.  We shall therefore assume $\pr{\bar{W}} = 0$ and
  arguing as in Proposition~\ref{prop:wreath-A4} then shows that
  \[
  d(W) = d(\bar{W}) = d(I_{\bar{W}}) \leq \max_{p}( d(\bar{A} \times
  A_{n}), h_{\bar{W}}(I_{p}))
  \]
  where the maximum is taken over all prime divisors~$p$
  of~$\order{A}$ with $p \nmid n$.  Lemma~\ref{lem:h-for-An} tells us
  that $h_{\bar{W}}(I_{p}) \leq d$ for all such~$p$.  It is easy to
  see that $d(\bar{A} \times A_{n}) = \max(d(\bar{A}),d(A_{n}))$
  \ (see, for example, \cite[Lemma~5.1]{Wiegold}) and hence $d(\bar{A}
  \times A_{n}) \leq d$.  This completes the proof.
\end{proof}

\section{Proof of the Main Theorem}

We now establish Theorem~\ref{thm:main}.  Accordingly, let $k \geq 2$
and let $G_{1}$,~$G_{2}$, \dots,~$G_{k}$ be a sequence of non-trivial
finite groups each of which is either an alternating group, a
symmetric group or a cyclic group.  Let $\Omega_{i}$~be the set upon
which the group~$G_{i}$ acts naturally and define
\[
W = G_{k} \wr G_{k-1} \wr \dots \wr G_{1}
\]
to be the iterated wreath product constructed with respect to these
actions.  Let us also define
\[
H = G_{k} \wr G_{k-1} \wr \dots \wr G_{2}
\]
and set $A = H/H'$ to be the abelianization of~$H$; that is,
\[
A = (G_{k}/G_{k}') \times (G_{k-1}/G_{k-1}') \times \dots \times (G_{2}/G_{2}').
\]
Set $\Gamma = \Omega_{k-1} \times \Omega_{k-2} \times \dots \times
\Omega_{1}$, which is the set upon which $G_{k-1} \wr \dots \wr G_{1}$~acts.

If $G_{k}$~is either an alternating group or a symmetric group of
degree~$n$ with $n \geq 5$, then its socle is a non-abelian simple
group and a standard argument shows that the copy of~$A_{n}^{\Gamma}$
in the base group of~$W$ is the unique minimal normal subgroup of this
wreath product.

Combining our work in the previous section with results of
Lucchini~\cite{Lucc97} establishes the number of generators required
for the wreath product of the abelian group~$A$ by our selected choice
of group~$G_{1}$.

\begin{lemma}
  \label{lem:keylemma-abelian}
  The number of generators of the wreath product~$A \wr G_{1}$ is
  given by
  \[
  d(A \wr G_{1}) = \begin{cases}
    \max (2, d(A), d_{3}(A)+1) &\text{if $G_{1} = A_{4}$}, \\
    \max (2, d(A)) &\text{if $G_{1} = A_{n}$ with $n \geq 5$}, \\
    \max (2, d(A), d_{2}(A)+1) &\text{if $G_{1} = S_{n}$ with $n \geq
      3$}, \\
    d(A) + 1 &\text{if $G_{1}$~is cyclic}.
  \end{cases}
  \]
\end{lemma}

\begin{proof}
  We have established the formula when $G_{1}$~is an alternating group
  of degree~$\geq 4$ in Propositions~\ref{prop:wreath-A4}
  and~\ref{prop:wreath-An}.  The formula when $G_{1}$~is a symmetric
  group of degree~$n \geq 3$ is part~(ii)
  of~\cite[Theorem~8]{Lucc97}.  Finally, the formula when $G_{1}$~is
  cyclic is a direct consequence of \cite[Corollary~4.4]{Lucc97}.
\end{proof}

The main step in the proof of Theorem~\ref{thm:main} is the following
lemma which reduces the number of generators of~$W$ to what is
established in Lemma~\ref{lem:keylemma-abelian}.

\begin{lemma}
  \label{lem:induction}
  Let $W$~and~$A$ be as defined above.  Then
  \begin{equation}
    d(W) = \max (2, d(A \wr G_{1})).
    \label{eq:reduction}
  \end{equation}
\end{lemma}

\begin{proof}
  We proceed by induction on~$k$.  Suppose first that $k = 2$.  If
  $G_{2}$~is the alternating group~$A_{4}$, the symmetric group
  $S_{3}$ or~$S_{4}$, or is cyclic, then $G_{2}$~is soluble with
  $d(G_{2}) \leq 2$ and
  Lemma~\ref{lem:Lucchini-facts}\ref{i:Lucc-solwreath} then gives the
  result.  If $G_{2}$~is alternating or symmetric of degree $n \geq
  5$, then $N = A_{n}^{\Omega_{1}} = (H')^{\Omega_{1}}$ is the unique
  minimal normal subgroup of~$W$ with quotient~$W/N$ isomorphic to
  $(H/H') \wr G_{1}$.  Theorem~\ref{thm:LuccMen} tells us that $d(W) =
  \max(2,d(W/N))$, which is Equation~\eqref{eq:reduction} in this
  case.

  Now assume that $k > 2$ and that the result holds for wreath
  products of alternating, symmetric or cyclic groups involving fewer
  than $k$~factors.  We first consider the case when $A$~is trivial;
  that is, when $H$~is perfect.  Thus $G_{2}$, \dots,~$G_{k}$ are all
  alternating groups of degrees at least~$5$.  Then the base group~$N
  = G_{k}^{\Gamma}$ is the unique minimal normal subgroup of~$W$ and
  so, by Theorem~\ref{thm:LuccMen},
  \[
  d(W) = \max(2, d(W/N)) = \max(2, d(G_{k-1} \wr \dots \wr G_{1})).
  \]
  Applying an induction argument, we deduce $d(W) = \max(d(G_{1}),2)$
  and so Equation~\eqref{eq:reduction} now follows in this case.

  For the remainder of the proof, we shall assume that $A \neq
  \trivsubgp$ and $H$~is not perfect.  Note then that $A \wr G_{1}$~is
  not cyclic and there is a homomorphism from~$W$ onto~$A \wr G_{1}$.
  Hence
  \[
  d(W) \geq d(A \wr G_{1}) = \max(2, d(A \wr G_{1}))
  \]
  and it remains to establish the reverse inequality.  We split into
  two cases according to the presentation rank of~$W$.  Suppose first
  that $\pr{W} = 0$.  Let $B$~be the abelianization of the iterated
  wreath product $G_{k} \wr G_{k-1} \wr \dots \wr G_{3}$.  Then, by
  induction,
  \[
  d(H) = \max(2, d(B \wr G_{2}) ).
  \]
  By Lemma~\ref{lem:keylemma-abelian}, the number of generators of $B
  \wr G_{2}$ is given by
  \[
  d(B \wr G_{2}) = \begin{cases}
    \max(2, d(B), d_{3}(B)+1 ) &\text{if $G_{2} = A_{4}$}, \\
    \max(2, d(B)) &\text{if $G_{2} = A_{n}$ with $n \geq 5$}, \\
    \max(2, d(B), d_{2}(B)+1 ) &\text{if $G_{2} = S_{n}$ with $n \geq
      3$}, \\
    d(B) + 1 &\text{if $G_{2}$~is cyclic}.
  \end{cases}
  \]
  Note that $A = B \times (G_{2}/G_{2}')$, so certainly $d_{p}(B) \leq
  d_{p}(A)$ for each prime~$p$.  Substituting this into the above
  formula yields $d(B \wr G_{2}) \leq d(A) + 1$ and hence
  \[
  d(H) \leq d(A \wr G_{1}) + 1
  \]
  by a further application of Lemma~\ref{lem:keylemma-abelian} to
  compute~$d(A \wr G_{1})$.

  Since $W$~is assumed to have zero presentation rank,
  \[
  d(W) = d(I_{W}) = \max \left( d(I_{A \wr G_{1}}), \left\lfloor
  \frac{d(I_{H}) - 2}{n} \right\rfloor + 2 \right)
  \]
  by Lemma~\ref{lem:Lucchini-facts}\ref{i:Lucc-augment}.  Certainly
  $d(I_{A \wr G_{1}}) \leq d(A \wr G_{1})$.  On the other hand, using
  $d(I_{H}) \leq d(H) \leq d(A \wr G_{1}) + 1$, we deduce that
  \[
  \left\lfloor \frac{d(I_{H})-2}{n} \right\rfloor + 2 \leq
  \left\lfloor \frac{d(A \wr G_{1}) - 1}{n} \right\rfloor + 2 \leq
  d(A \wr G_{1})
  \]
  since $n \geq 2$ and $d(A \wr G_{1}) \geq 2$ from our assumption
  that $A \neq \trivsubgp$.  We have therefore established $d(W) \leq
  d(A \wr G_{1})$ and hence Equation~\eqref{eq:reduction} in the case
  that $\pr{W} = 0$.

  Now assume that $W$~has positive presentation rank.  Consider first
  the case that $G_{k}$~is soluble (that is, either cyclic or one of
  the groups~$A_{4}$, $S_{3}$ or~$S_{4}$).  Then the base group $N =
  G_{k}^{\Gamma}$ is soluble and so
  \[
  d(W) = d(W/N) = d(G_{k-1} \wr G_{k-2} \wr \dots \wr G_{1})
  \]
  by Lemma~\ref{lem:Gruenberg}\ref{i:quot-pr}.  Hence, by induction,
  \[
  d(W) = \max(2, d(C \wr G_{1}))
  \]
  where $C$~is the abelianization of $G_{k-1} \wr G_{k-2} \wr \dots
  \wr G_{2}$.  There is a surjective homomorphism from $A \wr G_{1}$
  onto~$C \wr G_{1}$ induced from that of $A = (G_{k}/G_{k}') \times
  (G_{k-1}/G_{k-1}') \times \dots \times (G_{2}/G_{2}')$ onto~$C$ and
  hence
  \[
  d(W) \leq \max(2, d(A \wr G_{1})),
  \]
  as required.

  If $G_{k}$~is an alternating group of degree~$\geq 5$, then the base
  group $N = G_{k}^{\Gamma}$ is the unique minimal normal subgroup
  of~$W$ and so Theorem~\ref{thm:LuccMen} yields
  \[
  d(W) = d(W/N) = d(G_{k-1} \wr G_{k-2} \wr \dots \wr G_{1})
  \]
  and the required inequality then follows by the same argument.

  Finally, if $G_{k}$~is a symmetric group of degree~$n \geq 5$, then
  $W$~has a unique minimal normal subgroup~$N$ isomorphic to the
  direct product~$A_{n}^{\Gamma}$ and
  \[
  d(W) = d(W/N) = d(C_{2} \wr G_{k-1} \wr G_{k-2} \wr \dots \wr G_{1})
  \]
  by Theorem~\ref{thm:LuccMen}.  Thus, we have reduced to the case
  when $G_{k}$~is cyclic of order~$2$ which has already been
  addressed by an earlier case.  This completes the induction step and
  establishes the lemma.
\end{proof}

The main theorem, Theorem~\ref{thm:main}, now follows by combining
Lemmas~\ref{lem:keylemma-abelian} and~\ref{lem:induction}.

\bigskip

As a final illustration of our results, we present the following
example.  Comparing this example to the proof of
Lemma~\ref{lem:induction} gives some insight into the key steps within
it.

\begin{example}
  Consider the wreath product $W = C_{2} \wr C_{2} \wr C_{3} \wr
  A_{n}$ where $n$~is an integer with $n \geq 5$.  According to the
  statement of our main theorem, we need to make use of $H = C_{2} \wr
  C_{2} \wr C_{3}$ and its abelianization $A = C_{2} \times C_{2}
  \times C_{3} \cong C_{2} \times C_{6}$.  Then $d(A) = 2$ and hence
  $d(W) = 2$ by Theorem~\ref{thm:main}.

  The theorem also tells us how many generators are required for the
  group~$H$.  Indeed, $d(H) = \max(2,d(B)+1)$ where $B$~is the
  abelianization of~$C_{2} \wr C_{2}$; that is, $B = C_{2} \times
  C_{2}$.  Hence $d(H) = 3$.  In particular, this illustrates that we
  use the abelianization $A = H/H'$ rather than the subgroup~$H$ when
  determining the number of generators of~$W$.

  The proof of the theorem does not explicitly construct a set of
  generators for the iterated wreath product~$W$.  For this particular
  case, we describe in this example one pair of generators for this $W
  = C_{2} \wr C_{2} \wr C_{3} \wr A_{n}$ in the case that $n$~is an
  \emph{odd} integer with $n \geq 5$.  The calculation that follows
  indicates some of the subtleties within such wreath products that
  may appear hidden when simply applying the powerful theory.

  We shall view~$W$ as acting on a rooted tree where the root has
  $n$~edges to the vertices at level~$1$, each vertex at level~$2$ has
  three edges to the vertices at level~$2$, and each vertex at levels
  $3$ and~$4$ has two edges to the vertices in the level below.
  Figure~\ref{fig:tree} illustrates these trees for the case when $n =
  5$.  The vertices at level~$1$ will be labelled $1$,~$2$ and~$3$
  and, if $\alpha$~labels a vertex at some level~$k$ ($1 \leq k \leq
  3$), then its children will be labelled $\alpha1$, $\alpha2$ and, in
  the case that $k = 1$ only, $\alpha3$.  Thus a vertex of level~$k$,
  for $1 \leq k \leq 4$, is denoted by a word $a_{1}a_{2}\dots a_{k}$
  of length~$k$ where $a_{1} \in \{1,2,\dots,n \}$, \ $a_{2} \in
  \{1,2,3\}$ and $a_{3},a_{4} \in \{1,2\}$ (if $k$~is sufficiently
  large that these are required).  To define an element of~$W$, we
  shall first describe a permutation that should be applied to the
  vertices of level~$4$, then follow it by a permutation at level~$3$,
  then apply a permutation at level~$2$, and finally apply a
  permutation of the vertices at level~$1$.  We shall use the
  terminology ``applying a permutation~$\sigma$ at a vertex~$\alpha$''
  to mean that one permutes the children of~$\alpha$ according
  to~$\sigma$.

  Define $x$~to be the element of~$W$ that is given by applying the
  transposition~$(1\;2)$ at vertex~$11$ (that is, it first swaps
  vertex~$111$ with~$112$), then applying the $3$\nbd cycle
  $(1\;2\;3)$ at vertex~$5$, and finally applying the permutation
  $(1\;2)(3\;4)$ to the $n$~vertices below the root.  Define $y$~to be
  the element of~$W$ that applies $(1\;2)$~at vertex~$111$ and applies
  the $(n-2)$\nbd cycle $(2\;4\;5\,\dots\,n)$ at the root.
  Figure~\ref{fig:tree} illustrates these permutations for the case~$n
  = 5$ where we attach the permutation to each vertex to indicate how
  its descendants should be permuted.  We define~$L$ to be the
  subgroup of~$W$ generated by $x$~and~$y$.

  \begin{figure}
    \begin{center}
    \begin{tikzpicture}[
        inner sep=0.8pt,
        level distance=20pt,
        level 1/.style={sibling distance=80pt},
        level 2/.style={sibling distance=25pt},
        level 3/.style={sibling distance=12pt},
        level 4/.style={sibling distance=7pt}
      ]
      \node (root) [circle,fill,label={$(1\;2)(3\;4)$}] {}
      child {node (1) [circle,fill] {}
        child {node (11) [circle,fill,label={180:$(1\;2)$}] {}
          child {node (111) [circle,fill] {}
            child {node (1111) [circle,fill] {} }
            child {node (1112) [circle,fill] {} } }
          child {node (112) [circle,fill] {}
            child {node (1121) [circle,fill] {} }
            child {node (1122) [circle,fill] {} } } }
        child {node (12) [circle,fill] {}
          child {node (121) [circle,fill] {}
            child {node (1211) [circle,fill] {} }
            child {node (1212) [circle,fill] {} } }
          child {node (122) [circle,fill] {}
            child {node (1221) [circle,fill] {} }
            child {node (1222) [circle,fill] {} } } }
        child {node (13) [circle,fill] {}
          child {node (131) [circle,fill] {}
            child {node (1311) [circle,fill] {} }
            child {node (1312) [circle,fill] {} } }
          child {node (132) [circle,fill] {}
            child {node (1321) [circle,fill] {} }
            child {node (1322) [circle,fill] {} } } } }
      child {node (2) [circle,fill] {}
        child {node (21) [circle,fill] {}
          child {node (211) [circle,fill] {}
            child {node (2111) [circle,fill] {} }
            child {node (2112) [circle,fill] {} } }
          child {node (212) [circle,fill] {}
            child {node (2121) [circle,fill] {} }
            child {node (2122) [circle,fill] {} } } }
        child {node (22) [circle,fill] {}
          child {node (221) [circle,fill] {}
            child {node (2211) [circle,fill] {} }
            child {node (2212) [circle,fill] {} } }
          child {node (222) [circle,fill] {}
            child {node (2221) [circle,fill] {} }
            child {node (2222) [circle,fill] {} } } }
        child {node (23) [circle,fill] {}
          child {node (231) [circle,fill] {}
            child {node (2311) [circle,fill] {} }
            child {node (2312) [circle,fill] {} } }
          child {node (232) [circle,fill] {}
            child {node (2321) [circle,fill] {} }
            child {node (2322) [circle,fill] {} } } } }
      child {node (3) [circle,fill] {}
        child {node (31) [circle,fill] {}
          child {node (311) [circle,fill] {}
            child {node (3111) [circle,fill] {} }
            child {node (3112) [circle,fill] {} } }
          child {node (312) [circle,fill] {}
            child {node (3121) [circle,fill] {} }
            child {node (3122) [circle,fill] {} } } }
        child {node (32) [circle,fill] {}
          child {node (321) [circle,fill] {}
            child {node (3211) [circle,fill] {} }
            child {node (3212) [circle,fill] {} } }
          child {node (322) [circle,fill] {}
            child {node (3221) [circle,fill] {} }
            child {node (3222) [circle,fill] {} } } }
        child {node (33) [circle,fill] {}
          child {node (331) [circle,fill] {}
            child {node (3311) [circle,fill] {} }
            child {node (3312) [circle,fill] {} } }
          child {node (332) [circle,fill] {}
            child {node (3321) [circle,fill] {} }
            child {node (3322) [circle,fill] {} } } } }
      child {node (4) [circle,fill] {}
        child {node (41) [circle,fill] {}
          child {node (411) [circle,fill] {}
            child {node (4111) [circle,fill] {} }
            child {node (4112) [circle,fill] {} } }
          child {node (412) [circle,fill] {}
            child {node (4121) [circle,fill] {} }
            child {node (4122) [circle,fill] {} } } }
        child {node (42) [circle,fill] {}
          child {node (421) [circle,fill] {}
            child {node (4211) [circle,fill] {} }
            child {node (4212) [circle,fill] {} } }
          child {node (422) [circle,fill] {}
            child {node (4221) [circle,fill] {} }
            child {node (4222) [circle,fill] {} } } }
        child {node (43) [circle,fill] {}
          child {node (431) [circle,fill] {}
            child {node (4311) [circle,fill] {} }
            child {node (4312) [circle,fill] {} } }
          child {node (432) [circle,fill] {}
            child {node (4321) [circle,fill] {} }
            child {node (4322) [circle,fill] {} } } } }
      child {node (5) [circle,fill,label={45:$(1\;2\;3)$}] {}
        child {node (51) [circle,fill] {}
          child {node (511) [circle,fill] {}
            child {node (5111) [circle,fill] {} }
            child {node (5112) [circle,fill] {} } }
          child {node (512) [circle,fill] {}
            child {node (5121) [circle,fill] {} }
            child {node (5122) [circle,fill] {} } } }
        child {node (52) [circle,fill] {}
          child {node (521) [circle,fill] {}
            child {node (5211) [circle,fill] {} }
            child {node (5212) [circle,fill] {} } }
          child {node (522) [circle,fill] {}
            child {node (5221) [circle,fill] {} }
            child {node (5222) [circle,fill] {} } } }
        child {node (53) [circle,fill] {}
          child {node (531) [circle,fill] {}
            child {node (5311) [circle,fill] {} }
            child {node (5312) [circle,fill] {} } }
          child {node (532) [circle,fill] {}
            child {node (5321) [circle,fill] {} }
            child {node (5322) [circle,fill] {} } } } }
      ;
    \end{tikzpicture}
    \end{center}
    \begin{center}
      \begin{tikzpicture}[
        inner sep=0.7pt,
        level distance=20pt,
        level 1/.style={sibling distance=80pt},
        level 2/.style={sibling distance=25pt},
        level 3/.style={sibling distance=12pt},
        level 4/.style={sibling distance=7pt}
      ]
      \node (root) [circle,fill,label={90:$(2\;4\;5)$}] {}
      child {node (1) [circle,fill] {}
        child {node (11) [circle,fill] {}
          child {node (111) [circle,fill,label={180:$(1\;2)$}] {}
            child {node (1111) [circle,fill] {} }
            child {node (1112) [circle,fill] {} } }
          child {node (112) [circle,fill] {}
            child {node (1121) [circle,fill] {} }
            child {node (1122) [circle,fill] {} } } }
        child {node (12) [circle,fill] {}
          child {node (121) [circle,fill] {}
            child {node (1211) [circle,fill] {} }
            child {node (1212) [circle,fill] {} } }
          child {node (122) [circle,fill] {}
            child {node (1221) [circle,fill] {} }
            child {node (1222) [circle,fill] {} } } }
        child {node (13) [circle,fill] {}
          child {node (131) [circle,fill] {}
            child {node (1311) [circle,fill] {} }
            child {node (1312) [circle,fill] {} } }
          child {node (132) [circle,fill] {}
            child {node (1321) [circle,fill] {} }
            child {node (1322) [circle,fill] {} } } } }
      child {node (2) [circle,fill] {}
        child {node (21) [circle,fill] {}
          child {node (211) [circle,fill] {}
            child {node (2111) [circle,fill] {} }
            child {node (2112) [circle,fill] {} } }
          child {node (212) [circle,fill] {}
            child {node (2121) [circle,fill] {} }
            child {node (2122) [circle,fill] {} } } }
        child {node (22) [circle,fill] {}
          child {node (221) [circle,fill] {}
            child {node (2211) [circle,fill] {} }
            child {node (2212) [circle,fill] {} } }
          child {node (222) [circle,fill] {}
            child {node (2221) [circle,fill] {} }
            child {node (2222) [circle,fill] {} } } }
        child {node (23) [circle,fill] {}
          child {node (231) [circle,fill] {}
            child {node (2311) [circle,fill] {} }
            child {node (2312) [circle,fill] {} } }
          child {node (232) [circle,fill] {}
            child {node (2321) [circle,fill] {} }
            child {node (2322) [circle,fill] {} } } } }
      child {node (3) [circle,fill] {}
        child {node (31) [circle,fill] {}
          child {node (311) [circle,fill] {}
            child {node (3111) [circle,fill] {} }
            child {node (3112) [circle,fill] {} } }
          child {node (312) [circle,fill] {}
            child {node (3121) [circle,fill] {} }
            child {node (3122) [circle,fill] {} } } }
        child {node (32) [circle,fill] {}
          child {node (321) [circle,fill] {}
            child {node (3211) [circle,fill] {} }
            child {node (3212) [circle,fill] {} } }
          child {node (322) [circle,fill] {}
            child {node (3221) [circle,fill] {} }
            child {node (3222) [circle,fill] {} } } }
        child {node (33) [circle,fill] {}
          child {node (331) [circle,fill] {}
            child {node (3311) [circle,fill] {} }
            child {node (3312) [circle,fill] {} } }
          child {node (332) [circle,fill] {}
            child {node (3321) [circle,fill] {} }
            child {node (3322) [circle,fill] {} } } } }
      child {node (4) [circle,fill] {}
        child {node (41) [circle,fill] {}
          child {node (411) [circle,fill] {}
            child {node (4111) [circle,fill] {} }
            child {node (4112) [circle,fill] {} } }
          child {node (412) [circle,fill] {}
            child {node (4121) [circle,fill] {} }
            child {node (4122) [circle,fill] {} } } }
        child {node (42) [circle,fill] {}
          child {node (421) [circle,fill] {}
            child {node (4211) [circle,fill] {} }
            child {node (4212) [circle,fill] {} } }
          child {node (422) [circle,fill] {}
            child {node (4221) [circle,fill] {} }
            child {node (4222) [circle,fill] {} } } }
        child {node (43) [circle,fill] {}
          child {node (431) [circle,fill] {}
            child {node (4311) [circle,fill] {} }
            child {node (4312) [circle,fill] {} } }
          child {node (432) [circle,fill] {}
            child {node (4321) [circle,fill] {} }
            child {node (4322) [circle,fill] {} } } } }
      child {node (5) [circle,fill] {}
        child {node (51) [circle,fill] {}
          child {node (511) [circle,fill] {}
            child {node (5111) [circle,fill] {} }
            child {node (5112) [circle,fill] {} } }
          child {node (512) [circle,fill] {}
            child {node (5121) [circle,fill] {} }
            child {node (5122) [circle,fill] {} } } }
        child {node (52) [circle,fill] {}
          child {node (521) [circle,fill] {}
            child {node (5211) [circle,fill] {} }
            child {node (5212) [circle,fill] {} } }
          child {node (522) [circle,fill] {}
            child {node (5221) [circle,fill] {} }
            child {node (5222) [circle,fill] {} } } }
        child {node (53) [circle,fill] {}
          child {node (531) [circle,fill] {}
            child {node (5311) [circle,fill] {} }
            child {node (5312) [circle,fill] {} } }
          child {node (532) [circle,fill] {}
            child {node (5321) [circle,fill] {} }
            child {node (5322) [circle,fill] {} } } } }
      ;
    \end{tikzpicture}
    \end{center}
    \caption{Generators $x$~and~$y$ for $C_{2} \wr C_{2} \wr C_{3} \wr
      A_{5}$}
    \label{fig:tree}
  \end{figure}
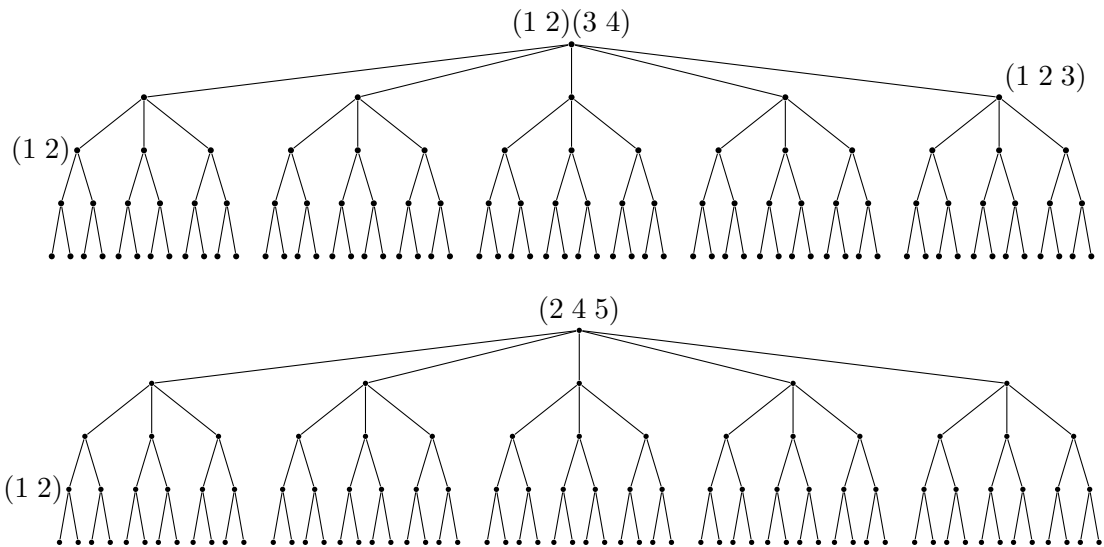

  Since $n$~is odd, $y$~is an element of~$W$ that has order~$2(n-2)$.
  Among its powers are the element that applies~$(1\;2)$ at the
  vertex~$111$ only and the element~$y'$ that applies the
  cycle~$(2\;4\;5\,\dots\,n)$ at the root.  Similarly, since $x$~is
  built from a $3$\nbd cycle acting below vertex~$5$ and permutations
  of $2$\nbd power order on a disjoint set of vertices, some power
  of~$x$ equals the element~$z$ that is given by applying $(1\;2)$~at
  vertex~$11$ followed by $(1\;2)(3\;4)$ at the root, and some power
  is the element which is the $3$\nbd cycle~$(1\;2\;3)$ at vertex~$5$.
  Since $A_{n}$~is generated by $(1\;2)(3\;4)$ and
  $(2\;4\;5\,\dots\,n)$, and  $L$~contains an element inducing a
  $3$\nbd cycle below one of the vertices at level~$1$, we deduce that
  $L$~contains elements that induce all permutations on the vertices
  of levels at most~$2$ by elements of the wreath product~$C_{3} \wr
  A_{n}$ (that is, $L$~projects onto the latter wreath product under
  the natural map).

  One calculates that $z^{2} = (111\;112)(211\;212)$ in cycle notation
  for a permutation of vertices at level~$3$; that is, $z^{2}$~is
  equal to the element of order~$2$ that swaps the vertices
  $111$~and~$112$ and the vertices~$211$~and~$212$.  Using
  permutations of the vertices in the levels above and then products
  of the resulting elements, we may therefore find in~$L$ all the
  permutations whose effect is to interchange the two pairs of
  descendants of vertices $\alpha$~and~$\beta$ of level~$2$ \ (that
  is, all $(\alpha1\;\alpha2)(\beta1\;\beta2)$ for $\length{\alpha} =
  \length{\beta} = 2$).  Finally, note that $xy'$~is given by
  applying~$(1\;2)$ at vertex~$11$ followed by an $n$\nbd cycle at the
  root.  Therefore $(xy')^{n}$~is a product of the $n$~transpositions
  that interchange the vertices~$a11$ and~$a12$ pairwise for $a \in
  \{1,2,\dots,n\}$.  As $n$~is odd and we can already produce any
  product of two transpositions of vertices at this level, we conclude
  that $L$~contains all transpositions of adjacent vertices in the
  third level and hence that $C_{2} \wr C_{3} \wr A_{5} \leq L$.
  Combining with the fact that some power of~$y$ is the transposition
  of the vertices $1111$~and~$1112$, we finally conclude that
  $x$~and~$y$ generate the wreath product $W = C_{2} \wr C_{2} \wr
  C_{3} \wr A_{5}$.
\end{example}

\end{document}